\documentclass[10pt]{amsart}
\usepackage{amssymb, amsthm,amsmath}
\usepackage{tikz}
\usepackage{calc}
\usepackage{fullpage}
\tikzstyle{vertex}=[circle, draw, inner sep=0pt, minimum size=3pt]
\newcommand{\vertex}{\node[vertex]}
\newcounter{Angle}

\author{Sean Griffin\\Dartmouth College}
\thanks{Funded by Paul K. Richter and Evalyn E. Cook Richter Memorial Fund}
\title{Minimal Pancyclicity}

\date{September 6, 2013}

\newtheorem{prop}{Proposition}
\newtheorem{clm}{Claim}
\newtheorem{thm}{Theorem}
\newtheorem{cor}{Corollary}

\theoremstyle{definition}
\newtheorem{conj}{Conjecture}
\newtheorem{defn}{Definition}
\newtheorem{ex}{Example}

\begin{document}

\begin{abstract}
A pancyclic graph is a simple graph containing a cycle of length $k$ for all $3\leq k\leq n$. Let $m(n)$ be the minimum number of edges of all pancyclic graphs on $n$ vertices. Exact values are given for $m(n)$ for $n\leq 37$, combining calculations from an exhaustive search on graphs with up to 29 vertices with a construction that works for up to 37 vertices. The behavior of $m(n)$ in general is also explored, including a proof of the conjecture that $m(n+1)>m(n)$ for all $n$ in some special cases.
\end{abstract}
\maketitle

The reader is referred to \cite{bondymurty} for terminology and notation used in this paper. A pancyclic graph with $n$ vertices is a graph with a cycle of every possible length, meaning it has a cycle of length $\ell$ for every $\ell$ such that $3\leq\ell\leq n$. Throughout this paper, the graph consisting of $n$ vertices forming a single cycle will be denoted $C_n$. Much of the work on pancyclic graphs centers around sufficient conditions for a Hamiltonian graph to be pancyclic.

\begin{thm}(Bondy \cite{PCI})
Let $G$ be a Hamiltonian graph and suppose $|E(G)|\geq n^2/4$ where $n=|V(G)|$. Then $G$ is either pancyclic or the complete bipartite graph $K_{n/2,n/2}$.
\end{thm}

\section{Pancyclic Graphs}

Let $G$ be a pancyclic graph with $n$ vertices and $m$ edges. Since $G$ contains a Hamiltonian cycle, a pancyclic graph can be viewed as a Hamiltonian graph (which will be represented as a circle in the diagrams) with $k=m-n$ chords through the Hamiltonian cycle. Constructions with a small number of edges have been discovered by George et al. in \cite{minpan}. A new construction with 5 chords (Figure \ref{fig:construction}) will be presented as an extension of their construction with 4 chords. It is easily checked that the graph has cycles of lengths 3 to 19 and of lengths $n-17$ to $n$, so for $G$ to be pancyclic it is required that $n-17\leq 20$ or $n\leq 37$. Since $21+x=n$ ($x$ as in the figure), $0\leq x\leq 16$.

\begin{figure}[h]
\centering
\begin{tikzpicture}[scale=2]
\foreach \i in {1,2,3,4,5,6,7,8,9,10,11,12,13,14,15,16,17,18,19,20,21,22} {
		\setcounter{Angle}{\i * 360 / 22}
		\vertex[fill] (d\i) at (\theAngle:1){};
	}
	\path 
		(d1) edge (d3)
		(d2) edge (d8)
		(d4) edge (d7)
		(d5) edge (d13)
		(d9) edge (d22)
	;
	\coordinate [label=right: $x$ edges] (dx) at (.99,.14);
	\draw (d1) arc (16.36:360:1);
	\draw[style=dotted] (d22) arc (0:16.36:1);

\end{tikzpicture}
\caption{Construction with $23$ to $37$ vertices}
\label{fig:construction}
\end{figure}
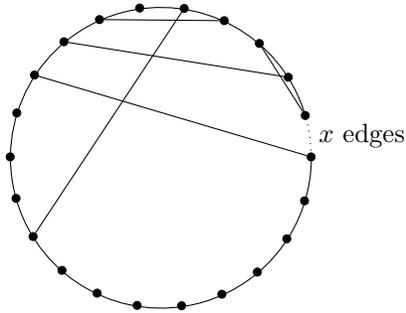

Let an $arc$ of a Hamiltonian graph be a maximal path contained in the Hamiltonian cycle such that none of the interior vertices have degree greater than 2. The trivial paths consisting of a single vertex are included in this definition of arc. Then there are exactly $2k$ many arcs in our Hamiltonian graph. Let us number the arcs of $H$, the Hamiltonian cycle, $A_1$,$A_2$,\dots,$A_{2k}$ in the clockwise direction. 

\begin{thm}(Shi \cite{shia})
For any set of chords $K$ in a Hamiltonian graph, there are at most 2 cycles including all chords in $K$ and discluding all other chords. These two cycles are (if they are cycles) $K\cup(\bigcup_{i=1}^k A_{2i})$ and $K\cup(\bigcup_{i=1}^k A_{2i-1})$. 
\end{thm}

\begin{cor}\label{shicor}
The number of cycles in a Hamiltonian graph with $k$ chords is at most $2^{k+1}-1$.
\end{cor}

This can be easily modeled in code to determine all cycles in a Hamiltonian graph $G$ when we add $k$ chords to a Hamiltonian cycle. An exhaustive search has been run for Hamiltonian graphs with at most 4 chords and for Hamiltonian graphs with 5 chords and at most 31 vertices. Since, by Corollary \ref{shicor}, there are at most 31 cycles in a Hamiltonian graph and there must be at least $n-2$ cycles (at least one for each cycle length), this suffices to show that no graphs on 25 or more vertices with 4 chords can be pancyclic. Therefore, this in combination with our construction with 5 chords gives us knowledge of the values of $m(n)$ for $n\leq 37$. Table \ref{mn table} records these results. All of these values agree with \cite{minpan}.

\begin{table}[ht]
	\centering
	\caption{Results for $m(n)$}
	
		\begin{tabular}{|c|c|c||c|c|c|}
			\hline
			n & k & m(n) & n & k & m(n)\\\hline
			3 & 0 & 3 & 21 & 4 & 25\\
			4 & 1 & 5 & 22 & 4 & 26\\
			5 & 1 & 6 & 23 & 4 & 27\\
			6 & 2 & 8 & 24 & 4 & 28\\
			7 & 2 & 9 & 25 & 5 & 30\\
			8 & 2 & 10 & 26 & 5 & 31\\
			9 & 3 & 12 & 27 & 5 & 32\\
			10 & 3 & 13 & 28 & 5 & 33\\
			11 & 3 & 14 & 29 & 5 & 34\\
			12 & 3 & 15 & 30 & 5 & 35\\
			13 & 3 & 16 & 31 & 5 & 36\\
			14 & 3 & 17 & 32 & 5 & 37\\
			15 & 4 & 19 & 33 & 5 & 38\\
			16 & 4 & 20 & 34 & 5 & 39\\
			17 & 4 & 21 & 35 & 5 & 40\\
			18 & 4 & 22 & 36 & 5 & 41\\
			19 & 4 & 23 & 37 & 5 & 42\\
			20 & 4 & 24 & & & \\
			\hline
			
		\end{tabular}
		\label{mn table}
\end{table}

Natural questions we might ask about $m(n)$ in general are establishing lower bounds, upper bounds, and its general behavior. In this paper, improvements on the bounds of $m(n)$ are explored, various properties of minimal pancyclic graphs are discussed, and the behavior of $m(n)$ in general is explored.
	
	\subsection{Bounds on $m(n)$}
	
	In \cite{PCI}, Bondy states bounds on $m(n)$ for general $n$ without proof. A proof of the lower bound has been included below:
	
	\begin{clm}(Bondy \cite{PCI})\label{bbounds}
	\[
	n + \log_2(n-1) - 1\leq m(n) \leq n + \log_2(n) + H(n) + O(1)
	\]
	where $H(n)$ is the smallest integer such that $(log_2)^{H(n)}(n)<2$ ($log_2$ applied $H(n)$ times).
	\end{clm}
	
	\begin{proof}(of lower bound)
	Let $G$ be a minimal pancyclic graph on $n$ vertices with $k$ chords, then the number of edges is $m(n)=n+k$. Since there must be at least $n-2$ cycles in $G$, by Corollary \ref{shicor},
	\[
	2^{k+1}-1\geq n-2
	\]
	which is equivalent to
	\[
	m(n)=n+k\geq n+log_2(n-1)-1
	\]
	\end{proof}

	It is clear that this lower bound can easily be improved by improving Shi's upperbound on the number of cycles in a Hamiltonian graph.
	
	\begin{thm}(Rautenbach and Stella \cite{rautenbach})
	Let $M(k)$ be the maximum number of cycles in a Hamiltonian graph with $k$ chords.
	
	\[
	M(k)\leq 2^{k+1}-1-k\left(\frac{\sqrt{k}-2}{\log_2(k)+2}-\frac{1}{4}\log_2(k)\right)
	\]
	\end{thm}
	
	Then $m(n)$ is at least $n+C$, where $C$ is the largest integer $k$ such that the expression in Theorem 1 is less than $n-2$. The reader is referred to \cite{sridharan} for constructions which give exact upper bounds on $m(n)$ for general $n$.
	
	\subsection{Properties of minimal pancyclic graphs}
	
	An upper bound on the maximum degree of a minimal pancyclic graph can now be established by considering how a large maximum degree affects the total number of cycles.
	
	\begin{prop}
	Let $G$ be a minimal pancyclic graph and let $\Delta$ be its maximum degree, and let $H(n)$ be as in Claim 1. If $\Delta > 4$ then there exists a positive constant $C$ such that
	
	\[
	\frac{2^\Delta+1}{\Delta^2+3\Delta+4}\leq C\left(\frac{n}{n-1}\right)2^{H(n)}
	\]
	\end{prop}
	
	\begin{proof}
	There must be a vertex $v$ in $G$ with degree $\Delta$. Then by Shi \cite{shia}, there is at most one cycle containing 2 chords incident with $v$, and there are no cycles containing 3 or more chords incident with $v$. Therefore, the number of cycles in $G$ is at most $2^{k+1}-1$ minus the number of potential cycles eliminated.
	
	\begin{eqnarray*}
	n-2 & \leq & 2^{k+1}-1-2^{k-\Delta}\left(\binom{\Delta}{2}+2\left(2^\Delta-\binom{\Delta}{0}-\binom{\Delta}{1}-\binom{\Delta}{2}\right)\right)\\
			& = & 2^{k+1}-1-\binom{\Delta}{2}2^{k-\Delta}-2^{k+1}+2^{k+1-\Delta}+\Delta 2^{k+1-\Delta}\\
			& = & 2^{k-\Delta-1}(\Delta^2+3\Delta+4)-1
	\end{eqnarray*}
	
	From this we get the inequality
	
	\[
	\frac{2^{\Delta+1}}{\Delta^2+3\Delta+4}\leq \frac{2^k}{n-1}
	\]
	
	and by using Claim \ref{bbounds} we get the desired inequality.
	\end{proof}
	
	\subsection{Behavior of $m(n)$}
	
	From Table \ref{mn table}, we see that for $n\leq 37$, $m(n)$ increases by at least 1 from $n$ to $n+1$ and increases by at most 2. Regarding the latter statement, we have the following easy result:
	
	\begin{prop}
	$m(n+1)\leq m(n) + 2$ for all $n\geq 3$
	\end{prop}
	
	\begin{proof}
	Let $G$ be a minimal pancyclic graph on $n$ vertices. Label the vertices $v_1,v_2,\dots,v_n$ in the order as they appear in the Hamiltonian cycle. We construct a new graph $G'$ from $G$ by adding the vertex $v_{n+1}$ to $V$ and the edges $v_{n+1}v_n$ and $v_{n+1}v_1$ to $E$. Then since $G$ is pancyclic, $G'$ has cycles of lengths 3 to $n$, and it has the cycle $v_1v_2\dots v_nv_{n+1}v_1$ of length $n+1$, so $G'$ is a pancyclic graph with $m(n)+2$ edges.
	\end{proof}
	
	From the former, we have the following conjecture:
	
	\begin{conj}
	$m(n)<m(n+1)$ for all $n\geq 3$
	\end{conj}
	
	This conjecture says that we need at least as many chords as a minimal pancyclic graph on $n$ vertices to create a pancyclic graph on $n+1$ vertices. This may seem very intuitive, but proving this statement is not as easy as one might imagine because the extremal graphs do not have many nice properties. However, we make partial progress toward a solution by looking at special cases where we can construct smaller pancyclic graphs from larger ones.
	
	Let all indices here-on be considered modulo $n$, where $n$ is the number of vertices.
	
	\begin{defn}
	Let $p=v_iv_j$ and $q=v_kv_\ell$ be two chords through a Hamiltonian cycle $v_1v_2\dots v_nv_1$. We say $p$ and $q$ \emph{cross} if these vertices are all distinct and in the order $v_i,v_k,v_j,v_\ell$ around the Hamiltonian cycle.
	\end{defn}
	
	\begin{prop}
	Let $G$ be a pancylic graph such that the vertices of its Hamiltonian cycle are $v_1,v_2,\dots,v_n$ in order. If $G$ satisfies the following the properties:
	
	\begin{enumerate}
		\item $G$ has more than one cycle of length 4.
		
		\item There exist two chords $p=v_i v_j$ and $q=v_{i+1}v_{j+1}$.
		
		\item The only chord in $G$ which crosses $p$ is $q$ and the only chord in $G$ which crosses $q$ is $p$.
	\end{enumerate}
	
	then $G$ is not minimal pancyclic.
	\end{prop}
	
	\begin{proof}
	We claim that $G'$, the subgraph of $G$ with $V(G')=V(G)$ and $E(G')=E(G)-p$ is pancyclic. For any cycle containing both $p$ and $q$, by condition 3 we can write the cycle in the form $v_j v_i S_1 v_{j+1} v_{i+1} S_2$ where $S_1$ and $S_2$ are sequences of vertices. Then the cycle $v_{i+1}v_i S_1 v_{j+1} v_j S^{-1}_2$, where $S^{-1}_2$ is $S_2$ in reverse order, has the same length and contains neither $p$ nor $q$. For any cycle containing $p$ but not $q$, it must also contain $v_i v_{i+1}$, so we can replace $v_j v_i v_{i+1}$ with $v_j v_{j+1} v_{i+1}$ to obtain a cycle containing $q$ but not $p$ with the same length. Therefore, $G'$ must also be pancyclic.
	\end{proof}
	
	\begin{defn}
	If $G$ is a Hamiltonian graph and $A$ is an arc in $G$ with nonzero length, $G_A$ is the graph obtained from $G$ by contracting a single edge in $A$. 
	\end{defn}
	
	\begin{thm}
	Let $G$ be pancyclic with $n>6$ vertices, Hamiltonian cycle $H$ and $k$ chords through $H$. (a) If there is an arc $A$ in $H$ such that there is no chord incident with both ends of arc $A$ with $|A|\geq (n-1)/2$, or (b) there is an arc $A$ such that there is a chord incident with both ends of $A$ with $|A|\geq (n+2)/3$, then $G_A$ is pancyclic.
	\end{thm}
	
	\begin{proof}
	(a) Suppose there is an arc $A$ such that there is no chord incident with both ends of $A$ and $|A|\geq (n-1)/2$. Suppose that $G_A$ is not pancyclic. Then there exists some integer $c$, $n\geq c \geq 3$, such that no cycle in $G$ not containing $A$ has length $c$ and no cycle containing $A$ has length $c+1$. Since $G$ is pancyclic, there exists a cycle containing $A$ with length $c$ and a cycle not containing $A$ with length $c+1$. Since we cannot create a cycle containing $A$ with length $|A|+1$ and the largest possible length of a cycle not containing $A$ is $n-|A|+1$, we have:
	
	\[
	c\geq |A|+2
	\]
	\[
	n-|A|+1 \geq c+1
	\]
	so
	\[
	n-|A|\geq c \geq |A|+2,
	\]
	and from our assumption on the length of $A$
	
	\[
	\frac{n}{2}-1\geq |A|\geq \frac{n-1}{2},
	\]
	a contradiction.
	
	(b) Let us now assume there is an arc $A$ such that there is a chord incident with both ends of $A$ with $|A|\geq (n+2)/3$. Suppose not, then there is again some $c$ such that no cycle not containing $A$ has length $c$, and no cycle containing $A$ has length $c+1$. We then have the inequalities:
	
	\[
	c\geq |A|+1
	\]
	\[
	n-|A|+1\geq c+1
	\]
	or
	\[
	n-|A|\geq c\geq |A|+1,
	\]
	and from our assumptions,
	\[
	\frac{n-1}{2}\geq |A|\geq \frac{n+2}{3}.
	\]
	If $n-|A|+1<2$ then $n=|A|$, a contradiction. Else, let \\$m=|A|+c-1$.
	
	\[
	2\leq n-|A|+1\leq n-\left(\frac{n+2}{3}\right)+1 = \frac{2n+1}{3} < 2\left(\frac{n+2}{3}\right)
	\]
	\[
	\leq 2|A| \leq |A|+c-1 = m \leq |A|-1+(n-|A|) = n-1 \leq n.
	\]
	Then $m$ is the length of a cycle in $G$, and such a cycle must contain $A$ since it is larger than $n-|A|+1$. Therefore, we can replace $A$ with the chord incident with both ends of $A$ to obtain a cycle not containing $A$ of length
	\[
	m-|A|+1 = c,
	\]
	a contradiction.
	
	\end{proof}
	
	\begin{cor}
	For $n>6$, if there is a minimal pancyclic graph on $n$ vertices with an arc of length at least $(n-1)/2$, then $m(n-1)<m(n)$. If there is a minimal pancyclic graph on $n$ vertices with an arc of length at least $(n+2)/3$, then $m(n-1)\leq m(n)$.
	\end{cor}
	
	\begin{proof}
	The first case is obvious. For the second case, we can construct a pancyclic graph on $n-1$ vertices 
	\end{proof}
	
	We know the bound $(n-1)/2$	is strict from the following example.
	\begin{ex}
	Figure \ref{fig:example} shows an example of a minimal pancyclic graph $G$ on $n=14$ vertices with 3 chords and an arc of length $n/2-1=6$ such that $G_A$ is not pancyclic.
	\end{ex}
	
	\begin{figure}[h]
\centering
\begin{tikzpicture}[scale=2]
\foreach \i in {1,2,3,4,5,6,7,8,9,10,11,12,13,14} {
		\setcounter{Angle}{\i * 360 / 14}
		\vertex[fill] (d\i) at (\theAngle:1){};
	}
	\path 
		(d1) edge (d13)
		(d3) edge (d14)
		(d9) edge (d14)
	;
	\draw (d14) arc (0:360:1);

\end{tikzpicture}
\caption{}
\label{fig:example}
\end{figure}

	In the next Proposition, we explore this example in general since it shows up in many minimal pancyclic graphs with 4 chords or less.
	
	\begin{prop}\label{propeven}
	Let $G$ be a pancyclic graph with even $n$ and an arc $A$ of length $n/2-1$ in its Hamiltonian cycle. Then either $G_A$ is a pancyclic graph or $G$ contains a subgraph that has cycles of all possible lengths except for 1 less than the number of vertices.
	\end{prop}
	
	\begin{proof}
	If $G_A$ is pancyclic, then we are done, so suppose $G_A$ is not pancyclic. Then there exists a $c$ such that no cycle containing $A$ has length $c+1$ and no cycle not containing $A$ has length $c$. We obtain the inequalities
	
	\[
	\frac{n}{2}+1\geq c\geq \frac{n}{2}+1
	\]
	so $c=n/2+1$. Let $G'$ be the subgraph of $G$ after deleting all edges and all vertices except for the ends from $A$. Then since there exists a cycle not containing $A$ with length $n/2+2$ and $G'$ is a graph of order $n/2+2$, $G'$ is a Hamiltonian graph. We know that $G$ contains cycles of all lengths less than $n/2+1$ which do not contain $A$, so $G'$ must have cycles will all of these lengths. However, by assumption we know that there is no cycle in $G$ not containing $A$ with length $n/2+1$, so the same is true for $G'$.
	\end{proof}
	
	\begin{cor}
	If there exists a minimal pancyclic graph with even $n$ and an arc of length $n/2-1$ in its Hamiltonian cycle, then either $m(n-1)<m(n)$ or $m(n/2+2)\leq m(n)+2-n/2$.
	\end{cor}
	
	\begin{proof}
	In the second case, we can add a single chord through the Hamiltonian cycle in $G'$ (as defined in the proof of Proposition \ref{propeven}) to obtain a pancyclic graph on $n/2+2$ vertices.
	\end{proof}
	
	\section{Acknowledgements}

I would like to give special thanks to the Paul K. Richter and Evalyn E. Cook Richter Memorial Fund. Without their aid, this research would not have been possible. I would also like to thank my research advisor Professor Sergi Elizalde for his guidance and his continuing support and to Richard Lange who provided me with valuable programs which were an essential part of the calculation of $m(n)$.

\end{document}